\documentclass[12pt]{amsart}
\usepackage{amsmath,amssymb,amsfonts,mathrsfs, amsthm}
\usepackage[margin=1 in,heightrounded=true,centering]{geometry}

\usepackage{amscd}
\usepackage{verbatim}
\usepackage{euscript}

\newcommand{\vep}{\varepsilon}

\DeclareMathOperator{\tr}{ tr }
\DeclareMathOperator{\divergence}{ div }
\DeclareMathOperator{\Rm}{Rm}
\DeclareMathOperator{\Rc}{Rc}
\DeclareMathOperator{\Sc}{S}

\newcommand{\Mbar}{\overline{M}}
\newcommand{\gbar}{\overline{g}}
\newcommand{\pd}[1]{\frac{\partial}{\partial #1}}
\DeclareMathOperator{\FP}{\underset{z=0}{FP}}
\DeclareMathOperator{\Res}{\underset{z=0}{Res}}

\newcommand{\del}{\partial}
\newcommand{\RenV}{\mathrm{RenV}}
\newcommand{\olg}{\overline{g}}
\newcommand{\frakc}{\mathfrak c}
\newcommand{\calO}{\mathcal O}

\newtheorem{theorem}{Theorem}[section]
\newtheorem{prop}[theorem]{Proposition}
\newtheorem{lemma}[theorem]{Lemma}
\newtheorem{remark}[theorem]{Remark}

\newtheorem{definition}[theorem]{Definition}
\newtheorem{maintheorem}{Theorem}
\newtheorem{maincorollary}[maintheorem]{Corollary}

\begin{document}

\title[Renormalized volume and Ricci flow]{Renormalized
volume and \\ the evolution of APEs}


\author{Eric Bahuaud}
\address{Department of Mathematics,
Seattle University,
901 12th Ave,
Seattle, WA 98122}

\email{bahuaude(at)seattleu.edu}

\author{Rafe Mazzeo}
\address{Department of Mathematics, Stanford University,
Stanford, CA 94305}

\email{mazzeo(at)math.stanford.edu}

\author{Eric Woolgar}
\address{Department of Mathematical and Statistical Sciences,
University of Alberta, Edmonton, Alberta, T6G 2G1, Canada}
\email{ewoolgar(at)ualberta.ca}

\date{}

\begin{abstract}
\noindent We study the evolution of the renormalized volume functional for even-dimensional asymptotically Poincar\'e-Einstein
metrics $(M,g)$ under normalized Ricci flow. In particular, we prove that
\[
\frac{d\,}{dt} \RenV(M^n, g(t)) = -\int_{M^n} \left (\Sc(g(t))+n(n-1) \right )\, dV_{g(t)},
\]
where $\Sc(g(t))$ is the scalar curvature for the evolving metric $g(t)$. This implies that if
$\Sc+n(n-1) \geq 0$ at $t=0$, then $\RenV(M^n,g(t))$ decreases monotonically.
For odd-dimensional
asymptotically Poincar\'e-Einstein metrics with vanishing obstruction tensor, we find that
the conformal anomaly for these metrics is constant along the flow.

We apply our results to the Hawking-Page phase transition in black hole thermodynamics. We compute renormalized volumes for the
Einstein $4$-metrics sharing the conformal infinity of an AdS-Schwarzschild black hole. We compare these to the free energies
relative to thermal hyperbolic space, as originally computed by Hawking and Page using a different regularization technique,
and find that they are equal.

\end{abstract}

\maketitle

\section{Introduction}
\setcounter{equation}{0}
\noindent \noindent There is a well-known connection between the Riemannian geometry of noncompact Poincar\'e-Einstein
(PE) manifolds and the conformal geometry on their asymptotic boundaries, which are compact manifolds of one lower
dimension.  While relevant to the AdS/CFT correspondence in string theory \cite{Maldacena, Witten, HS, PS, dHSS}, it traces
back to earlier mathematical studies by Fefferman and Graham \cite{FG}, \cite{FG2}, who were motivated by the program
of classifying conformal invariants on the boundary in terms of the `interior' Riemannian geometry of the PE space. PE
geometry has appeared independently in various other mathematical guises too. Finally, this has other physical roots,
related to the thermodynamics of anti-de Sitter black holes \cite{HP}.

An interesting quantity in this setting is the renormalized volume of a PE space $(M^n,g)$. This is defined
by an Hadamard regularization scheme of the volumes of certain special families of compact subdomains which
exhaust $M$. A remarkable theorem, due to Henningson-Skenderis \cite{HS} and Graham-Witten \cite{GW},
see also \cite{Gr}, asserts that this renormalized volume, $\mathrm{RenV}\,(M,g)$, is well defined when
$n$ is even in the sense that it is independent of the choice of certain geometrically natural exhausting
sequence of subdomains (i.e., those for which the boundaries form equidistant families). When $n$ is odd,
one obtains a quantity no longer independent of choices, but which depends on these choices in a simple and
comprehensible way. All of this is explained in great detail in \cite{Gr}, but is recalled below.  For other, distinct forms of renormalized volume, see \cite{BC, HJS, KS}.

In this paper we study a class of spaces $(M^n, g)$ which are {\it asymptotically Poincar\'e-Einstein}
(APE) in a strong asymptotic sense.  It is still possible in this setting to define $\RenV(M,g)$
and show that it has the same invariance properties as for PE metrics.
Our goal here is to consider the behavior of these APE spaces and of their renormalized volumes
under the Ricci flow. We derive a formula for the time derivative of $\RenV(M, g(t))$, where $g(t)$
is a solution to the Ricci flow equations, and show that in certain circumstances this quantity
is monotone. We give some applications and conclude by explaining the relevance of this circle
of ideas to black hole thermodynamics. There is a well-known example of a $4$-manifold
where the same conformal structure at infinity is induced by three non-isometric PE metrics in the interior \cite{HP}.
We argue in \S 4 that the detailed consideration of phase transitions in black hole thermodynamics leads to consideration
of Ricci flow of manifolds with APE metrics and their renormalized volumes.

We shall consider the normalized Ricci flow
\begin{equation}
\begin{split}
\del_t g &=  -2 \left ( \Rc(g)+(n-1)g\right ) :=-2E(g)\ ,\ t\in[0,T) \\
g(0) & =g^0\ ,
\end{split}
\label{eq1.1}
\end{equation}
assuming that the initial metric $(M,g^0)$ is APE. Requiring that the flow remain asymptotically hyperbolic and fixing the conformal
infinity may be thought of as a ``boundary condition'' at infinity.
The short-time existence for this flow for general smooth conformally compact initial data
was obtained in \cite{Bahuaud}, see also \cite{AlbinAldanaRochon}. Another proof of short-time existence specifically
adapted to the APE setting appears in \cite{QSW}.

While the long-time behavior of conformally compact manifolds under Ricci flow is certainly no simpler than
that in the compact case, one can obtain strong control at spatial infinity in any finite time slice in this class of
spaces. Our first result is that the APE condition is preserved under the flow.  This result holds in
both even and odd dimensions.
\begin{maintheorem}\label{TheoremA}
Let $(M^n,g^0)$, $n \geq 2$, be APE, and let $g(t)$ be a solution of \eqref{eq1.1}.
Then $(M, g(t))$ remains APE on some possibly smaller interval $[0,T_0)$.
\end{maintheorem}
\noindent This theorem differs only slightly from the existence result in \cite{QSW} in that for those authors
the approximately Poincar\'e-Einstein condition for the initial metric $g^0$ and the resulting flow $g(t)$ is phrased
purely in terms of the decay of the Einstein tensor $E(g)$ at $\del M$, whereas our result also posits that $g^0$ has a
complete asymptotic expansion and this is preserved under the flow. The APE condition is that some initial
set of terms in this expansion satisfy a particular differential relationship. The fact that this is preserved under
the flow means that these leading terms are constant along the flow. In a sequel to this paper we shall study the evolution under
Ricci flow of asymptotically hyperbolic metrics which are asymptotic to Poincar\'e-Einstein metrics in a much
weaker sense, and where the variation of the individual terms in the expansion under the flow becomes one of central technical
issues. It is highly likely that $T_0 = T$, i.e., that the solution remains APE in the maximal interval of existence, but we do
not address this here.

It is natural to study how various numerical quantities associated to the family of metrics $g(t)$ evolve with $t$. To that end, we prove
\begin{maintheorem}\label{TheoremB} Suppose that $\dim M$ is even and $(M,g(t))$ is as in Theorem \ref{TheoremA}.
Then for $t\in[0,T_0)$, the renormalized volume $\mathrm{RenV}[g(t)]$ satisfies
\begin{equation}
\frac{d\, }{dt}\mathrm{RenV}((M,g(t))) =-\int_M \left ( \Sc(g(t)) + n(n-1) \right ) \, dV_{g(t)}\equiv -\int_M tr^{g(t)} E(g(t))
\,dV_{g(t)}\ , \label{eq1.2}
\end{equation}
and
\begin{equation}
\begin{split}
\frac{d^2\,}{dt^2} \RenV((M,g(t)) &= - \int_M \left( 2\left \vert E \right \vert^2 -2(n-1)\tr^g E-\left ( \tr^gE\right )^2\right) \, dV_{g(t)}\\
&\equiv - \int_M \left [ 2\left \vert Z \right \vert^2 -\frac{(n-2)}{n}\left ( \tr^gE\right )^2-2(n-1)\tr^g E\right ] \, dV_{g(t)}.
\label{2ndderiv}
\end{split}
\end{equation}
Here $\Sc(g(t))$ is the scalar curvature of $g(t)$, $\tr^g E(g(t)) = \Sc(g(t))+n(n-1)$, and $Z(g(t))$ is the trace-free part of the Ricci tensor.
\end{maintheorem}
Note that if $\Sc(g(t)) + n(n-1) = 0$ at any value $t=t_0$, so that $\RenV$ is stationary under the Ricci flow there,
then
\begin{equation}
\label{eq1.4} \frac{d^2\,}{dt^2}\bigg \vert_{t=t_0} \RenV(M,g(t)) = - 2\int_M \left \vert Z(g(t_0)) \right \vert^2 dV_{g(t_0)}\ .
\end{equation}

The expression \eqref{eq1.2} is closely related to a well-known formula of Anderson \cite[eqn. (0.9)]{Anderson}
for the variation of renormalized volume for a one-parameter family $g(t)$ of PE metrics (where, in particular, the
conformal infinity varies). Strictly speaking, \cite{Anderson} only contains details for the four dimensional case,
although the proof extends to all even dimensions; we refer to \cite{Albin} for a perhaps more transparent
proof based on different computations which works directly in all even dimensions. This formula asserts that
\begin{equation}
\frac{d\, }{dt}\mathrm{RenV}((M,g(t))) = -\frac14 \int_{\del M} \langle \kappa_0, g_{n-1} \rangle_{g_0}
\,dA_{g(t)}\ ; \label{eqand}
\end{equation}
here $\kappa = g'(0)$ is the `Jacobi field' associated to the variation, and $\kappa_0$ is the leading
term in its expansion as $x \searrow 0$. The fact that this formula localizes at the boundary
is due to the fact that the variation is only amongst PE metrics. One very interesting point is that the term $g_{n-1}$
is globally determined by the metric $g(0)$; it is the analogue of the Neumann data of the metric in this
setting. There is, of course, a more general formula which includes both interior and boundary terms,
which calculates the variation for families of metrics which are not PE and for which the conformal infinity changes.

We also note the following closely related setting. There is a theory developed by Graham and Witten \cite{GW} for
renormalized areas and volumes for complete, properly embedded minimal submanifolds in PE spaces, of arbitrary
dimension and codimension, which have regular asymptotic boundaries. The special case of renormalized area for such
two-dimensional minimal surfaces $Y \subset \mathbb H^3$, studied at length in \cite{AM}, bears a striking resemblance
to the special case of the considerations here concerning renormalized volume of four-dimensional PE spaces. In particular,
there is a local formula for renormalized area in that setting which mirrors the formula \eqref{eq4.1} for renormalized
volume in four dimensions below.  The formul\ae\ in \cite{AM} for the first and second variations of renormalized area
are once again entirely localized to the asymptotic boundary at infinity because the variation is amongst minimal
surfaces in a hyperbolic manifold; there would be interior terms if either the surfaces are not minimal or if the
ambient metric is not exactly Einstein. One other point is that in the present setting, when $n$ is even,
the APE condition precludes the term at order $n-1$ in the expansion of the volume form (this is explained in the next
section), which would be the one appearing in the boundary term.

In one forthcoming paper, we shall treat a more delicate situation concerning the Ricci evolution of a metric with a
general even expansion (but which is not necessarily APE) up to order $n$. In yet another paper, the first two authors
study the variation of renormalized volume for a hypersurface evolving by mean curvature flow.

Continuing on, applying the maximum principle to $\Sc(g)+n(n-1)$ gives our
\begin{maintheorem}\label{TheoremC}
If $(M,g(t))$ is as in Theorem \ref{TheoremB} and $\Sc(g^0) + n(n-1) \ge 0$, then  $\Sc(g(t)) + n(n-1) \ge 0$ for all $t$
and $\mathrm{RenV}(g(t))$ is monotone decreasing. Furthermore, $\RenV(g(t))$ is constant on an interval $t\in I$ iff
$\Sc(g(t))+n(n-1)=0$ for all $0 \leq t < T$.
\end{maintheorem}
As one application, we record a {\it no breathers} theorem:
\begin{maincorollary}\label{CorollaryD}
Let $(M,g(t))$ be as in Theorem \ref{TheoremB} and suppose that $\Sc(g^0) + n(n-1)\ge 0$.
If there exist times $0 \leq t_1 <t_2$ such that $(M,g(t_1))$ is isometric to $(M,g(t_2))$,
then $(M,g(t))$ is a stationary solution and Einstein.
If, in addition, $n<8$ and the conformal class induced on $\partial_\infty M$ by $g^0$ is that of the
standard round sphere, then $(M,g(t))$ is isometric to hyperbolic space $\mathbb H^n$ for all $t$.
\end{maincorollary}

The organization of this paper is as follows. In \S 2 we describe the various asymptotic conditions for metrics,
including conformally compactifiable, asymptotically hyperbolic, and APE, and prove Theorem \ref{TheoremA}.
Next, in \S 3, we recall the alternative Riesz regularization method to define $\mathrm{RenV}$ and
prove Theorem \ref{TheoremB}. The proofs of Theorem \ref{TheoremC} and Corollary \ref{CorollaryD} 
appear here too. Finally, \S 4 describes how the APE evolution of renormalized volume appears in the thermodynamics of
$4$-dimensional black holes, where renormalized volume equals the Hawking-Page difference of (formally divergent) actions on-shell.
Our monotonicity theorem in the APE setting echoes numerical results of Headrick and Wiseman \cite{HP} obtained in the
different setting of ``black holes in a finite cavity''.  The APE setting with its strict monotonicity formula and absence of a finite
boundary makes the Headrick-Wiseman monotonicity argument rigorous. An interesting question is whether the Headrick-Wiseman
free energy diagram interpretation applies to the off-shell renormalized volume (see \S 4).

We thank Gerhard Huisken for first posing the question which motivated this work. RM acknowledges support by National Science Foundation grant DMS--1105050. EB was supported in part by a new faculty
startup grant at Seattle University. EW was supported by a Discovery Grant from the Natural Sciences and Engineering Research
Council of Canada. EB thanks Robin Graham for useful conversations. EW is grateful to Don Page for a discussion of the
Hawking-Page transition. All three authors also offer thanks to the Park City Math Institute 2013 Summer Program, where
this work was completed. We thank Toby Wiseman for commenting on a draft of this paper.

\section{The evolution of expansions}
\setcounter{equation}{0}

\noindent The renormalization scheme defining $\mathrm{RenV}$ requires that the metric $g$ admits
a particular type of asymptotic expansion near $\del M$. We begin by recalling some facts about
such expansions and then show that they are preserved under the normalized Ricci flow.

To fix notation, let $M^n$ be an open manifold which is the interior of a compact manifold with boundary $\Mbar$.
A metric $g$ on $M$ is called conformally compact if there exists a smooth defining function $x$
for $\del M$ such that $\overline{g} = x^2 g$ extends to a smooth metric on $\Mbar$.
The restriction of $\olg$ to $T\del M$ determines a conformal class $\frakc(g)$ on the boundary,
called the conformal infinity of $g$. An asymptotically hyperbolic (AH) metric is one which is conformally
compact and satisfies $|dx|^2_{\olg} = 1$ on $\del M$. Observe that neither $x$
nor $\olg$ are individually well-defined since $g = x^{-2} \olg = (a x)^{-2} (a^2 \olg)$ for any positive
smooth function $a$, but the AH condition is independent of this ambiguity.  It is also useful to
consider conformally compact metrics with various other regularity conditions, for example
metrics for which $\olg$ and $x$ are both polyhomogeneous rather than smooth.

An AH metric $g$ is Poincar\'e-Einstein (PE) if it is also Einstein, i.e.\ if  the Einstein tensor $E(g) := \Rc(g)
+ (n-1)g$ vanishes identically. Any such metric is a stationary solution of the normalized Ricci flow.
If $g^0$ is AH, then it is proved in \cite{Bahuaud} that the solution $g(t)$ of the normalized Ricci
flow with this initial data remains AH, and in particular, smoothly conformally compact, for a short time.
In the optimal situation, this solution $g(t)$ should converge to a PE metric, but of course this may
fail to happen because of the development of singularities. It is still unclear whether there is any
possibility of new and specifically boundary singularities forming, or whether singularity formation
is confined entirely to the interior, where it can be studied by the same techniques as in the
compact case.

If $g$ is AH, then a theorem of Graham and Lee \cite{GL} states that for each representative $g_0$ of $\frakc(g)$,
there is a uniquely defined special boundary defining function $x$ such that
\begin{equation} \label{eq-ah-decomp}
g = \frac{dx^2 + \hat{g}_x}{x^2}
\end{equation}
near $\del M$. Here $\hat{g}_x$ is a smooth family of symmetric $2$-tensors on $\del M$, which can be expanded
in a Taylor series as
\begin{equation}
\hat{g}_x = g_0 + x g_1 + \ldots + x^n g_n + \ldots
\label{expgenmet}
\end{equation}
This choice of special boundary defining function is a coordinate gauge which is very useful for
many of the algebraic computations concerning these metrics. For example, it is immediate from \eqref{eq-ah-decomp}
and \eqref{expgenmet} that if $g$ is AH, then it is asymptotically Einstein to order $1$ in the sense that
$|E(g)|_g = \calO(x)$. On the other hand, if $g$ is actually PE, then it was shown by Fefferman and
Graham \cite{FG} that if $g_0$ and (perhaps surprisingly) also $g_{n-1}$ are fixed, then all the remaining
coefficients $g_j$ are formally determined by the Einstein condition $E(g) = 0$. In particular,
the leading coefficients up to order $n-2$ are determined formally by $g_0$ alone. To be more precise, there is a
difference between even and odd dimensions, since in the latter case the terms at order $n-1$ and higher
typically include logarithmic factors. The general form of a PE metric is:
\begin{eqnarray}
g=\begin{cases}x^{-2}\big ( dx^2+g_{0}+x^2g_{2} +\dots +x^{n-2}g_{n-2}& \mbox{ } \\
\qquad +x^{n-1}g_{n-1}+ \calO(x^n)\big )& \mbox{for $n$ even, } \\
\\
x^{-2}\big ( dx^2+g_{0}+x^2g_{2}+\dots
+x^{n-3}g_{n-3}& \mbox{ } \\ \qquad +x^{n-1}\left ( g_{n-1}+ \tilde{g}_{n-1}\log x\right )
+\calO(x^{n}\log x)\big ) & \mbox{for $n$ odd, }
\end{cases}
\label{expg}
\end{eqnarray}
where each $g_{j}$, $0 \leq j \leq n-2$ is obtained by applying a universal differential expression to $g_0$.
When $n$ is even, $\tr^{g_0} g_{n-1} = 0$. When $n$ is odd, the so-called ambient-obstruction tensor $\tilde{g}_{n-1}$
is also determined by $g_0$ and satisfies $\tr^{g_0} \tilde{g}_{n-1} = 0$; and $\tr^{g_0}g_{n-1}$ is formally determined from
the prior coefficients.

This leads to the
\begin{definition}\label{definition2.1}  An AH metric $g$ is called asymptotically Poincar\'e-Einstein (APE) if $|E(g)|_g = \calO(x^n)$.
\end{definition}

In light of the remarks above, in order to work with smooth conformal compactifications when $n$ is odd, we
require that the conformal infinity be obstruction flat, i.e. $\tilde{g}_{n-1} = 0$.  We refer to a metric satisfying this condition
as an unobstructed APE.
\begin{prop} \label{prop:APE-expansion}\label{proposition2.2}
If $g$ is APE, then the expansion of $\hat{g}_x$ in the Graham-Lee normal form for $g$, relative to any choice of
representative $g_0 \in \frakc(g)$, has the form \eqref{expg}.
\end{prop}
%
\begin{remark}\label{remark2.3}
The key point here is that the leading coefficients in the metric expansion are all determined by $g_0$ and are the same as the corresponding coefficients of the expansion of a PE metric $g'$ for which
$g_0 \in \frakc(g')$. Note, however, that there is no guarantee that a PE metric with the conformal
infinity $[g_0]$ necessarily exists!
\end{remark}
\begin{proof}[Sketch]
Since this result is only a small modification of a well known one when $g$ is PE, we merely sketch the proof,
following the argument in \cite{GrahamHirachi}.

First observe that in any smooth coordinate system $(x, y_1, \ldots, y_{n-1})$, the APE condition implies that
the components of the Einstein tensor in this coordinate system satisfy
\begin{equation}
\label{eq2.4} E(g)_{ij} = \calO(x^{n-2}).
\end{equation}

Now calculate $E(g)= \Rc(g) + (n-1) g$ using \eqref{eq-ah-decomp} and \eqref{expg}.  We use Greek indices to label
components on the boundary and $0$ for $\partial_x$; a `prime' denotes the
derivative with respect to $x$, and $\nabla$ is the Levi-Civita connection for $g_x$ for fixed $x$.
To avoid a proliferation of notation, we will suppress use of the hat notation from equation \eqref{eq-ah-decomp}.
From \cite{GrahamHirachi} we obtain
\begin{align}
\nonumber 2x E_{\alpha\beta} &= -x g_{\alpha\beta}'' + x g^{\mu\nu} g_{\alpha \mu}' g_{\beta\nu}'
- \frac{x}{2} g^{\mu\nu} g_{\mu\nu}'g_{\alpha\beta}' + (n-2) g_{\alpha\beta}'
+ g^{\mu\nu}g_{\mu\nu}' g_{\alpha\beta} + 2x \Rc(g_x)_{\alpha\beta}. \\
\label{eq2.5} E_{\alpha 0} &=E_{0\alpha}= \frac{1}{2} g^{\mu\nu}\left( \nabla_{\nu} g_{\alpha\nu}'
- \nabla_{\alpha} g_{\mu\nu}' \right). \\
\nonumber E_{00} &= -\frac{1}{2} g^{\mu\nu} g_{\mu\nu}''
+ \frac{1}{4} g^{\mu\nu} g^{\lambda\xi} g_{\mu\lambda}' g_{\nu\xi}'
+ \frac{1}{2} x^{-1} g^{\mu\nu} g_{\mu\nu}'.
\end{align}
Differentiating the first equation $s-1$ times with respect to $x$, $s \leq n-2$, and then setting $x=0$ gives
\begin{eqnarray}
\label{eq2.6}
&&(n-1-s) \partial_x^s g_{\alpha\beta} + g^{\mu\nu} \partial_x^s g_{\mu\nu} \cdot g_{\alpha\beta}\\
\nonumber
&=& \partial_x^{s-1} (2x E_{\alpha\beta})_{x=0} + (\mbox{terms containing} \; \partial_x^k g_{\alpha\beta}\; \mbox{with} \; k < s).
\end{eqnarray}
We now formally solve for $g_x$ given the equation $|E| = O(x^n)$.  By keeping careful track of the parity with respect to
$x$ in equation \eqref{eq2.6}, we may solve for the leading part of the expansion of $g_x$ in terms of $g_0$.  Careful analysis
requires separate arguments depending on the parity of $n$.  When $n$ is even, we obtain an even expansion determined by
$g_0$ to order $n-2$.  When $s=n-1$, we find that $\partial_x^{n-2} (2x E_{\alpha\beta}) = \calO(x)$.  Together with a parity argument,
we see that the trace-free part of $g_{n-1}$ must vanish too.  Observe that all of these computations are insensitive to whether
$g$ is PE or simply APE.

When $n$ is odd, we obtain an even expansion determined by $g_0$ up to order $n-3$.  Parity considerations on the right hand
side of \eqref{eq2.6} force the inclusion of a logarithmic term at order $n-1$.  One finds that $\tilde{g}_{n-1}$
is trace-free and determined by $g_0$, whereas the trace of $g_{n-1}$ is determined, but the trace-free part is undetermined.

\end{proof}

As noted earlier, it is proved in \cite{Bahuaud} that if $g^0$ is a smooth AH metric,
then the solution $g(t)$ of normalized Ricci flow with initial condition $g^0$
remains AH, and moreover, $\frakc(g(t)) = \frakc(g^0)$, for all $t$ in the maximal
interval of existence of the solution.   Even though the conformal infinity remains
constant, so we can choose the same representative $g_0$ for all $t$, the
special boundary defining function $x$ depends on the interior metric $g(t)$
as well as $g_0$, and hence depends on $t$.  This means that the Graham-Lee
normal form for $g(t)$ evolves in a rather complicated way, and cannot be
expressed simply in terms of a single boundary defining function.  This is the
chief difficulty in understanding the variation of the renormalized volumes
$\RenV(g(t))$.

As a first step toward that, we recall a result from \cite{QSW} which asserts that
pointwise decay at $\del M$ of the Einstein tensor persists under the flow.
\begin{theorem}[Lemma 4.3 of \cite{QSW}] \label{theorem2.4}
Let $g(t)$ be a solution to the normalized
Ricci flow for $0 \leq t < T$, with $g(0) = g^0$ an AH metric satisfying
\begin{equation}
\label{eq2.7} \left \vert \Rm(g^0) \right \vert \leq k_0, \quad \mbox{and}\qquad \left \vert \nabla \Rm(g^0)\right \vert \leq k_1
\end{equation}
for two constants $k_0, k_1 > 0$. Suppose too that
\begin{equation}
\label{eq2.8} \left \vert E(g^0)\right \vert_{g^0} \leq C_0 x^{\gamma}, \quad \mbox{and} \qquad \left \vert \nabla E(g^0)\right \vert_{g^0} \leq C_0 x^{\gamma}
\end{equation}
for some $\gamma > 0$. Then there exists $C = C(k_0,k_1, n, C_0, T) > 0$ such that
\begin{equation}
\label{eq2.9} \left \vert E(g(t))\right \vert_{g(t)} \leq C x^{\gamma}, \quad \left \vert \nabla E(g(t))\right \vert_{g(t)} \leq C x^{\gamma}, \quad \left \vert \nabla^2 E(g(t))\right \vert_{g(t)}
\leq \frac{C}{\sqrt{t}} x^{\gamma}.
\end{equation}
\end{theorem}

\begin{proof}[Sketch]
As before, we give only a brief indication, outlining the proof for the estimate of $E = E(g(t))$, and
referring to \cite[Lemma 4.3]{QSW} for more details.

First compute the evolution equation for $|E|^2$ along the flow:
\begin{equation}
\label{eq2.10} \partial_t |E|^2 = \Delta |E|^2 - 2 |\nabla E|^2 + 4 \Rm_{ijkl} E^{il} E^{jk}.
\end{equation}
Shi's well-known estimate \cite{Shi} bounds the curvature $\Rm(g(t))$ by a time dependent constant.  Rescaling $|E|$
by setting $\overline{E} = x^{-n} E$ (this notation will be used in this proof only), then there is a new evolution equation
which involves derivatives of $x$, which must be bounded as well. After some work, we obtain
\begin{equation}
\label{eq2.11}
\partial_t |\overline{E}|^2 \leq \Delta |\overline{E}|^2  + C(T) |\overline{E}|^2.
\end{equation}
Now apply the Ecker-Huisken maximum principle \cite{EH} \emph{mutatis mutandis} to conclude that $|\overline{E}|$
is bounded for some short time; this means simply that $g(t)$ remains APE.
\end{proof}

This result makes no use the fact that $g(t)$ has an expansion for $t > 0$. However, coupling this theorem
with the results of \cite{Bahuaud} and Proposition \ref{prop:APE-expansion}, we obtain our first main theorem,
which we restate more precisely as
\begin{theorem} \label{theorem2.5}
Assume that $g^0$ is an APE metric, and further that $g^0$ has vanishing obstruction tensor if $n$ is odd.
Let $g(t)$ be the corresponding solution to the normalized Ricci flow.
Then $g(t)$ remains APE on some time interval $0 \leq t < T_0$. We may thus write
\begin{equation}
 g(t) = \frac{dx_t^2 + \hat{g}_{x_t}}{x_t^2},\label{eq2.12}
\end{equation}
where $x_t$ is the special boundary defining function associated to $g_0$ and $g(t)$, and
\begin{eqnarray}
g=\begin{cases}x_t^{-2}\big ( dx_t^2+g_{0}+x_t^2g_{2} +\dots +x_t^{n-2}g_{n-2}& \mbox{ } \\
\qquad +x_t^{n-1}g_{n-1}(t)+ \calO(x_t^n)\big )& \mbox{for $n$ even, } \\
\\
x_t^{-2}\big ( dx_t^2+g_{0}+x_t^2g_{2}+\dots
+x_t^{n-3}g_{n-3}& \mbox{ } \\ \qquad +x_t^{n-1} g_{n-1}(t)
+\calO(x_t^{n})\big ) & \mbox{for $n$ odd, }
\end{cases}
\end{eqnarray}
where $\tr^{g_0} g_{n-1}(t) = 0$ when $n$ is even, and $\tr^{g_0} g_{n-1}(t)$ is independent of time when $n$ is odd.
\end{theorem}
This expansion plays a prominent role in the next section.

As a final comment here, recall that we have been using the result from \cite{Bahuaud} that if $g^0$
has a smooth asymptotic expansion (and in particular if it is APE), then its Ricci evolution $g(t)$
also has a smooth expansion so long as the flow exists.  If $g(t)$ were to exist for all $t > 0$ and converge
to a PE metric, then there would have to be a `jump' in the expansion at $t=\infty$ when $n$ is odd
in order to capture the extra log terms in \eqref{expg}. This is an interesting effect, though one which may be
difficult to analyze precisely. Examples of jumps in the asymptotic structure of the limit metric have previously been observed in Ricci flows on asymptotically conical surfaces \cite{IMS} and, for $n\ge 3$, rotationally symmetric asymptotically flat Ricci flows \cite{OW}.

\section{Renormalized Volume and the Ricci flow}
\setcounter{equation}{0}

\noindent  We now turn to the renormalized volume functional.  Although we noted in \S 1 that
$\RenV$ is defined using Hadamard regularization, we begin by providing an alternative definition
using Riesz regularization. This is entirely equivalent, as explained in \cite{Albin}, and we refer to
that paper for more details.  With the exception of the final Proposition, we assume throughout
this section that $M$ is even-dimensional.

If $g$ is APE, then
\begin{equation}
\label{eq3.1}
\begin{split}
dV_g &= \left( \frac{\det( \hat{g}_x )}{\det(g_0)} \right)^{\frac{1}{2}} \frac{dV_{g_0} dx}{x^n}\\
&= \left( 1 + v_2 x^2 + \; \mbox{even terms} \; + v_{n} x^n + \cdots \right) \frac{dV_{g_0} dx}{x^n};
\end{split}
\end{equation}
where the $v_i$ , $0 \leq i \leq n-1$, are locally determined functions on $\partial M$ which have been studied
closely due to their connection with the $Q$-curvature function, see \cite{CFG}.

Using this expansion, it is straightforward that
\begin{equation}
\label{eq3.2} z \mapsto \zeta_x(z) = \int_M x^z dV_g
\end{equation}
extends meromorphically from $\mbox{Re}\,(z) > n-1$ to the whole complex plane with simple poles at
$\zeta_j = n-1-j$, $j = 0, 1, 2, \ldots$.  The Riesz regularization of volume is then defined as the
{\it finite part} of this function at $z = 0$:
\begin{equation}
\label{eq3.3} \RenV(M,g) := \FP \; \; \zeta_x(z).
\end{equation}

In the following, write
\begin{equation}
\label{eq3.4} x_t =: e^{\omega_t} x\ ,
\end{equation}
and let an overdot denote $\pd{t}$. Note that we are not evaluating this
derivative at $t=0$, which is a departure from the calculations in \cite{Albin}.

We now begin the proof of Theorem \ref{TheoremB}. To study the renormalized volume under the Ricci flow we compute
\begin{align}
\label{eq3.5} \pd{t} \FP \; \int_M x_t^z \, dV_{g_t} &= \FP \int_M z x^{z-1}_t \dot{x} dV_{g_t} + \FP \int_M x^z \left( \frac{1}{2}
\tr^{g(t)} \dot{g} \right) \, dV_{g_t}.
\end{align}

The second term is easier to understand. Indeed, contracting the flow equation
\begin{equation}
\label{eq3.6} \partial_t g = - 2 (\Rc(g) + (n-1) g)
\end{equation}
with $g$ shows that $\tr^{g} \dot{g} = -2(\Sc + n(n-1))$ and by the APE condition, this has pointwise
norm decaying like $\calO(x^n)$. Since $dV_g \sim x^{-n}$, this shows that
\begin{equation}
\label{eq3.7} \FP \int_M x^z \left( \frac{1}{2} \tr^{g(t)} (\partial_t g) \right) \, dV_{g_t}= - \int_M (\Sc + n(n-1))\,  dV_g.
\end{equation}

Now turn to the first term in \eqref{eq3.5}. We will show this integral vanishes. Because several constants appearing
here depend on $t$, this proof only shows that the result only holds for $t$ in a sufficiently small interval,
perhaps much less than the overall time of existence of the solution. For simplicity, we still denote the
time of existence by $T$ but shall later set the value of $T$ to be small. In any case,  all of the compactified
metrics $\gbar_t$ are uniformly equivalent for $0 \leq t \leq T$ and
\begin{equation}
\label{eq3.8} 0 < c_T \leq e^{\omega_t} \leq C_T
\end{equation}
for some $c_T < C_T$.

We require some knowledge about the expansion of $\omega_t$ with respect to the original boundary
defining function $x$.
\begin{lemma}\label{Lemma3.1}
If $g(t)$ is APE for $0 \leq t \leq T$, then
\begin{equation}
\label{eq3.9} \dot \omega_t(x,y,t) = \dot\omega_{(n)}(y,t) x^n + \calO_T(x^{n+1});
\end{equation}
here the notation $\calO_T$ indicates that the constants depend on $T$.
\end{lemma}
\begin{proof}
Suppose inductively that
\begin{equation}
\omega_t = \omega_{(k)}(y,t) x^k + \calO_T(x^{k+1}).
\label{induct}
\end{equation}
Since $\omega_t(0,y,t) = 0$, we may take $k \geq 1$ initially.

Now
$d \omega_t = \omega_{(k)}(y,t) k x^{k-1} dx + \calO_T(x^{k})$,
and hence
$d \dot\omega_t = \dot \omega_{(k)}(y,t) k x^{k-1} dx + \calO_T(x^{k})$,
where the error terms are measured in pointwise norm with respect to any one of the compactified metrics
$\olg(t)$.  Following \cite[Lemma 5.2]{Albin}, differentiating the relation
\begin{equation}
\label{eq3.11} 1 = |dx_t|^2_{x^2_t g_t} = x^{-2} {g_t} (dx, dx) + 2x^{-1} {g_t} (dx, d\omega_t) + {g_t} (d\omega_t, d\omega_t)
\end{equation}
yields, after some rearrangement, that

\begin{align}
\label{eq3.12} g_t(dx,d\dot\omega_t) +  x g_t(d \omega_t, d\dot\omega_t) &
= x^{-1} E_t(dx, dx) + 2 E_t(dx, d\omega_t) + x E_t(d\omega_t, d\omega_t).
\end{align}
We rewrite this further using the compactified metrics as
\begin{align}
\label{eq3.13} \gbar_t(dx,d\dot\omega_t) &= x^{-1} x^2_t E_t(dx, dx)
+ 2 x_t^2 E_t(dx, d\omega_t) + x x_t^2 E_t(d\omega_t, d\omega_t)
- x \gbar_t(d \omega_t, d\dot\omega_t).
\end{align}
Recall that $C_1 x \leq x_t \leq C_2 x$ and
\begin{equation}
\label{eq3.14} |E_t(dx,dx)|_{g_t} \leq |E_t|_{g_t} |dx|^2_{g_t} = \calO(x^{n-2}).
\end{equation}
Hence the terms involving $E_t$ are all $\calO(x^{n-1})$, while the final term on the right is $\calO(x^{2k-1})$.

Finally, using the inductive hypothesis for $\omega_t$ on the left, then
\begin{equation}
\label{eq3.15} \dot \omega_{(k)} = \calO(x^{n-k}) + O(x).
\end{equation}
Thus for $k < n$ we deduce that $\dot\omega_{(k)} = 0$, so \eqref{induct} holds with $k$ replaced by $k+1$.
\end{proof}

Now return to the first integral in the main computation. We find that
\begin{equation}
\label{eq3.16}
\begin{split}
& \quad \FP  \int_M z x^{z-1}_t \dot{x} dV_{g_t} = \FP \int_M z x^{z}_t \dot{\omega}_t dV_{g_t} = \Res \int_M x^{z}_t \dot{\omega}_t dV_{g_t} \\
&= \Res \left[ \left( \int_0^{\vep} \int_{\partial M} x^{z}_t \dot{\omega}_t
\left( 1 + v_2 x^2 + \; \mbox{even terms} \; +v_{n} x^n + \cdots \right) \frac{dV_{g_0} dx}{x^n} \right) + C(\vep) \right] \\
&= 0.
\end{split}
\end{equation}
The key point in this final equality is that $\dot \omega_t$ cancels the singularity in the volume measure, so the residue vanishes.
This concludes the proof of formula (\ref{eq1.2}).

To derive the second variation formula (\ref{2ndderiv}), we use
the evolution formula for $\tr^{g(t)}E(g(t))$, which is
\begin{equation}
\label{eq3.17}
\left ( \frac{\partial}{\partial t}-\Delta \right ) \textrm{tr}^{g(t)}E(g(t))
=2\vert E(g(t))\vert^2-2(n-1)\textrm{tr}^{g(t)}E(g(t))\ .
\end{equation}
Differentiating (\ref{eq1.2}) along the flow yields
\begin{equation}
\label{eq3.18}
\frac{d^2\, }{dt^2}{\RenV}(M,g(t))=-\int_M \left [ \frac{\partial }{\partial t}\left ( \tr^{g}E\right )
+\frac12 \left ( \tr^g E\right ) g^{ij} \frac{\partial g_{ij}}{\partial t} \right ] \, dV_g\ .
\end{equation}
Plugging \eqref{eq3.17} into \eqref{eq3.18} and simplifying, we obtain that
\begin{equation}
\label{eq3.19} \frac{d^2\, }{dt^2}{\RenV}(M,g(t))
=-\int_M \left [ 2\left \vert E \right \vert^2 -2(n-1)\left ( \tr^g E\right ) -\left ( \tr^g E\right )^2\right ]\, dV_g\ .
\end{equation}
There is a term at $\del_\infty M$ in this integration by parts which vanishes because the APE condition implies
that $\tr^g E$ and $x\del_x \tr^g E$ are both  $\calO(x^n)$. This concludes the proof of Theorem \ref{TheoremB}.

Theorem \ref{TheoremC} is proved, as in \cite{BW}, by noting that if $\tr^{g(0)}E(0) \geq 0$, then the maximum principle yields
positivity of $E(t)$ so long as the flow exists.

\begin{prop}\label{Proposition3.2}
If $g(t)$ is an AH solution to the Ricci flow, and if $\tr^{g(0)}E(0) \geq 0$ then $\tr^{g(t)}E(t) \geq 0$.
\end{prop}
\begin{proof}
Recall that if $g$ is AH, then
\begin{equation}
\label{eq3.20} \tr^g E(g) = \Sc(g) + n(n-1) = \calO(x),
\end{equation}
which evolves according to equation (\ref{eq3.17}).  Note that this is much weaker decay than that of an APE metric.

Suppose by way of contradiction that
$\displaystyle \inf_{p \in M} \tr^{g(t)}E(g(t))(p,t_*) <0$
for some $t_* < T$.   There is a constant $C = C(t_*)$ such that
\begin{equation}
\label{eq3.21}\left \vert \tr^{g(t)}E(g(t))\right \vert \equiv \left \vert \Sc(g(t))+ n(n-1)\right \vert \leq C\, x.
\end{equation}

The classical parabolic minimum principle implies that
$\tr^{g(t)}E(g(t))$ cannot attain a negative minimum in the
region $\{ x > 0 \} \times [0,t_*]$.  Indeed, at such a point,
$\Delta \tr^{g(t)}E(g(t))\ge 0$,
so the right side of
equation (\ref{eq3.17}) is strictly positive, while the left side is nonpositive.

This leaves the possibility that a negative minimum for $\tr^{g(t)}E(g(t))$ occurs at $x=0$.  Choose an exhaustion of $M \times [0,t_*]$
by compact sets of the form $M_k ~=~\{~x~\geq~\frac{1}{k}~\} \times [0,t_*]$.  We know that $\Sc(g) + n(n-1)$ cannot
attain a negative minimum in the interior $M_k$, and must be at least nonnegative on its boundary, at $x = 1/k$. Taking a limit,
we see that $\tr^{g(t)}E(g(t))\geq 0$ everywhere,
which is a contradiction.
\end{proof}

Using this in \eqref{eq1.2} gives monotonicity, and hence completes the proof of the main statement of Theorem \ref{TheoremC}.
Finally, from Proposition \ref{Proposition3.2},  since $\tr^{g(t)}E(g(t))$ is nonnegative for $t\in (t_1,t_2)$, if it is not everywhere
zero, then by \eqref{eq1.2} $\RenV((M,(g(t_2))<\RenV((M,(g(t_1))$. This proves the final statement of Theorem \ref{TheoremC}.

\begin{proof}[Proof of Corollary \ref{CorollaryD}]
If $(M,g(t_1))$ and $(M,g(t_2))$ are isometric, then their renormalized volumes are equal. Since $\Sc(g^0)\ge -n(n-1)$,
Theorem \ref{TheoremC} gives that $\Sc(g(t)) + n(n-1) \equiv 0$. But the evolution equation \eqref{eq3.17} for $\tr^{g(t)}E(g(t))$
implies that the full Einstein tensor $E(g(t))\equiv 0$ for all $t$, so that $g(t)$ is Einstein and stationary.

Finally, by Qing's rigidity theorem \cite{Qing}, any conformally compactifiable Einstein manifold of dimension $n<8$
whose conformal infinity is a round sphere must be ${\mathbb H^n}$.
\end{proof}

We conclude this section with a few remarks about the odd-dimensional case, referring the reader to \cite{Gr} for more detail.
The regularization scheme outlined above now produces a log term in the volume expansion, called the conformal anomaly,
which is a conformal invariant of the boundary conformal class.  The renormalized volume is still the constant term in the
expansion however this term is no longer conformally invariant.  Indeed given two representatives $g_0$, $g_0'$ of $\frakc(g)$,
where $g_0' = e^{2 \Upsilon} g_0$ for a smooth conformal factor $\Upsilon$, the difference in renormalized volumes computed
with respect to the special boundary defining function associated to each conformal representative is now a boundary integral
of a polynomial nonlinear differential operator in $\Upsilon$.

We conclude this section with an easy consequence of our expansions.

\begin{prop}\label{RemarkConformalAnomaly}\label{Lemma3.3}
When $n$ is odd and $(M, g^0)$ is an unobstructed APE, then the conformal anomaly term in the expansion
of $g(t)$ is constant along the flow.
\end{prop}
\begin{proof}
From \cite{Gr}, the conformal anomaly is $\int_{\partial M} v_{(n-1)} \, dV_{g_0}$.  The conclusion follows since the volume coefficients
up to $v_{(n-1)}$ are independent of $t$ by Theorem \ref{theorem2.5}.
\end{proof}

\section{Discussion}
\setcounter{equation}{0}

\noindent The black holes that participate in the Hawking-Page phase transition \cite{HP} of black hole thermodynamics
are Poincar\'e-Einstein manifolds. Below, we compute the renormalized volumes of these black holes (in $4$-dimensions) and find that the renormalized volumes equal the free energies (relative to thermal hyperbolic space, see below). This is remarkable in that the free energies are computed with a regularization method that requires a global subtraction of two metrics, while renormalized volumes are computed via the intrinsic regularization of Hadamard or, equivalently, of Riesz.

Consider $S^2\times S^1$ with metric $\gamma$, the product of the
standard (curvature equal to $+1$) metric on $S^2$ and a circle of length $\beta$. There are three well-known
Poincar\'e-Einstein metrics for which $(S^2\times S^1, [\gamma])$ is the conformal infinity. Two of these,
the so-called {\it large} and {\it small black hole} metrics, $g_{\mathrm{lbh}}$ and $g_{\mathrm{sbh}}$, are warped
product metrics on the bulk manifold $M_1 = S^2 \times \mathbb R^2$, while the third, $g_h$, called
\emph{thermal hyperbolic space} in physics, is the hyperbolic metric on $M_2 = B^3 \times S^1$, realized as
the quotient of $\mathbb H^4$ by a hyperbolic dilation with translation distance determined by $\beta$.
The parameter $1/\beta$ is called the {\it temperature} (in units of the Boltzmann constant). It is related to the temperature measured by a static (thus non-inertial, accelerated) observer at a point in the manifold by a redshift \cite[equation (2.4)]{HP}.

Now recall the formula of Anderson \cite{Anderson}, see also \cite{Albin}, for renormalized volume of
PE spaces
\begin{equation}
\label{eq4.1}\RenV(M,g) = \frac{4\pi^2}{3}\chi(M)-\frac{1}{24} \int_M \left ( \vert {\rm Rm} \vert^2 -4|Z|^2-24 \right ) dV_g.
\end{equation}
As an aside, this formula generalizes immediately to APE spaces and, as explained in the introduction, our formula \eqref{eq1.2} for $\del_t \RenV$ in $4$ dimensions can also be obtained
by differentiating \eqref{eq4.1}, as in \cite{Anderson}. Equation (\ref{eq1.2}) can also be obtained via a closely related version of this calculation which uses an identity due to Berger \cite{Berger} along with the APE condition to
eliminate the boundary terms arising in subsequent integrations by parts. Here $\chi(M)$ is the Euler characteristic
and $Z$ the tracefree Ricci tensor. We note that \cite{Anderson} uses different conventions so the constants
in that paper are different.

Using \eqref{eq4.1} we obtain that
\begin{equation}
\label{eq4.2} \RenV(M_2, g_h) = 0,
\end{equation}
since all the terms on the right of (\ref{eq4.1}) vanish for this metric, while
\begin{equation}
\label{eq4.3}
\RenV(M_1, g_{\mathrm{lbh/sbh}})= \frac{8\pi^2}{3} \frac{a_\pm^2(1-a_\pm^2)}{(1+3a_\pm^2)}.
\end{equation}
The black hole horizon radii $a_- < a_+$ are the two roots of the quadratic equation (see \cite[equation (2.7)]{HP})
\begin{equation}
\label{eq4.4} 3a^2-\frac{4\pi}{\beta}a+1=0 \ .
\end{equation}
In fact there are two distinct real roots only when $\beta^2>\frac{4\pi^2}{3}$, which we
always assume. In any case, when $\beta < \pi$, which corresponds to high temperature, then
\begin{equation}
\label{eq4.5} \RenV(M_1, g_{\mathrm{lbh}}) < \RenV(M_2, g_h) = 0 < \RenV(M_1, g_{\mathrm{sbh}})
\end{equation}
On the other hand, thermal hyperbolic space is the minimizer of $\RenV$ in the low temperature regime $\beta>\pi$.

Up to an overall normalization to account for physical units, \eqref{eq4.3} (or equally, the difference between \eqref{eq4.3}
and \eqref{eq4.2}) is precisely the formula \cite[eqn (2.9)]{HP} for the regularized action, computed there using the ad hoc regularization procedure of taking the difference between the gravitational action of a black hole with a distance cut-off and that of thermal hyperbolic space, also with a cut-off, and then taking the limit in which the cut-off goes to infinity. Furthermore, a simple calculation using (\ref{eq4.4}) and the relation $a^3+a-2m=0$ between the horizon radius and the mass $m=m_{\mathrm{lbh/sbh}}$ of the (large or small) black hole allows us to re-write the right hand side of (\ref{eq4.3}). This yields the familiar thermodynamic relation
\begin{equation}
\label{eq4.6} \frac{3}{8\pi\beta}\RenV(M_1, g_{\mathrm{lbh/sbh}})=\langle E \rangle -\frac{1}{\beta} S\ ,
\end{equation}
where $S=\pi a_{\pm}^2$ is the entropy and $\langle E \rangle$, the expectation value of the internal energy, can be computed by a standard relation in the canonical ensemble which yields $\langle E \rangle= m_{\mathrm{lbh/sbh}}$. Thus, in this example at least, the renormalized volume functional has a thermodynamic interpretation. By evaluating it on these black hole metrics and multiplying the results by $3/8\pi\beta$, one can find both the free energy liberated in the phase transition and the energy barrier to be surmounted in order for the transition to proceed.

We next discuss the role played by APE flows. Equation \eqref{eq1.4} shows that a PE manifold is marginally linearly stable in the space of APE perturbations. (By contrast, $\RenV$ is strictly unstable at any non-Einstein constant scalar curvature metric). It is well-known \cite{Prestidge} that the small black hole metric $g_{\mathrm{sbh}}$ is linearly unstable in the sense
that its Lichn\'erowicz Laplacian $\Delta_L$ has a negative eigenvalue, and hence it is also linearly unstable for the Ricci flow, since the linearized Ricci-DeTurck flow is $\partial_t -\Delta_L$, see \cite{HW}. Thus, a small perturbation of this PE metric produces a nearby metric with the same renormalized volume to second order and which is not Einstein, which then evolves under the full (normalized) Ricci flow \eqref{eq1.1} so that its renormalized volume strictly decreases.

The case of black holes with boundary, which has been studied numerically \cite{HW}, presents an analogous picture.
Prescribing appropriate Dirichlet conditions, the boundary can be filled by two different black holes, and also by so-called
hot flat space (${\mathbb R}^3\times S^1$ with a flat metric). In that case, the Einstein-Hilbert action is largest at the
small black hole, and this is (linearly) unstable for Ricci flow. The numerical flows can be followed all the way from
a small perturbation of the small black hole to either the large black hole or to hot flat space. In the latter space
there must be a surgery to account for the change in topology. The flow is a gradient flow, but with respect to a metric
of indefinite signature, so existence of monotonic quantities is not generally expected. Nonetheless, monotonicity of the
Einstein-Hilbert action is observed numerically along these flows, and is used to construct what is referred to
in \cite{HW} as a ``novel'' free energy diagram for the phase transition in that setting.

Our analytical results are consistent with these numerical and thermodynamic arguments, in the setting of APE manifolds rather than manifolds with boundary. The APE flow is also a gradient flow (whose potential is the renormalized volume) with respect to an indefinite metric on the space of APE metrics. And
although the metric is still indefinite, Theorem \ref{TheoremC} ensures that the renormalized volume is monotone along the flow.

In the introductory section, we briefly mentioned certain mathematical generalizations on which we are working. We now close by taking note of some physics questions that this discussion raises, and which we hope will motivate future work.

The Hawking-Page transition occurs in all dimensions, and plays an important role in the AdS/CFT correspondence \cite{Witten},
which naturally applies in $5$ dimensions. However, renormalized volume appears to differ from the Hawking-Page regularization \cite[equation (2.18)]{Witten} in $5$ dimensions. Certain formulas facilitate the computation of $\RenV$ in all even
dimensions, see \cite{Albin} and \cite{CQY}. It would be interesting to compare the Hawking-Page free energy
to $\RenV$ in these dimensions. Also, in the case of toral black holes, where a Hawking-Page transition \cite{SSW} has also been observed \cite{SSW}, one can easily compute ${\RenV}$ in all dimensions, and indeed in that case the Hawking-Page regularized action equals ${\rm RenV}$ in all dimensions.

Headrick and Wiseman \cite{HW}, in their study of Schwarzschild black holes in a finite cavity, have postulated that Ricci flow leads to a novel free energy diagram for black hole phase transitions. This was based on numerical evidence for a monotonic quantity (the action), combined with a gradient flow argument taking into account the indefinite metric signature. Our setting places the monotonicity argument on a rigorous analytical footing, because it replaces the finite boundary of the cavity with the APE asymptotic condition. However, to carry through the full Headrick-Wiseman argument, we would need to compute the renormalized volume \emph{off-shell}. Then equation (\ref{eq4.4}) does not hold---instead, these black hole metrics have cone points in the ${\mathbb R}^2$ factor. We were unable to carry this through. It would be very interesting to do so, either analytically or numerically. As well, it would of course be very interesting to study numerically the Ricci flow of APE metrics for initial data near the small black hole.

Finally, many phase transitions admit a modern description in terms of renormalization group flows. For example, the
ferromagnetic 2-dimensional Ising model has a phase transition described by ``block spin renormalization'' which
interpolates between the macroscopic ordered and disordered states of the ferromagnet. It is certainly suggestive
that the dynamics of the Hawking-Page phase transition may have a Ricci flow description. This would fit with the well-known
fact that Ricci flow is, to one loop, the renormalization group flow for a 2-dimensional sigma model,
and hence also approximately for closed strings \cite{Friedan}, the fundamental excitations thought to give rise to collective states such as black holes.

\end{document}